\documentclass[svgnames]{amsart}
\usepackage[english]{babel}
\usepackage[utf8]{inputenc}

\usepackage{svg}
\usepackage{hyperref}
\usepackage{enumitem}
\usepackage{amssymb, amsfonts, amsmath, amsthm, amscd, amsxtra, latexsym, mathabx,mathtools}

\usepackage{tikz-cd}
\usepackage{epigraph}

\theoremstyle{plain}
\newtheorem{thm}{Theorem}[section] 
\newtheorem{lemma}[thm]{Lemma}
\newtheorem{cor}[thm]{Corollary}
\newtheorem{prop}[thm]{Proposition}
\newtheorem{claim}[thm]{Claim}

\newtheorem{thmintro}{Theorem}

\newtheorem{corintro}[thmintro]{Corollary}

\theoremstyle{definition}
\newtheorem{defn}[thm]{Definition}
\newtheorem{rem}[thm]{Remark}
\newtheorem{notation}[thm]{Notation}
\newtheorem{remintro}[thmintro]{Remark}
\newtheorem{question}[thmintro]{Question}
\newtheorem{ex}[thm]{Example}
\newtheorem{exintro}[thmintro]{Example}

\newenvironment{claimproof}{\begin{proof}}{\end{proof}}

\newcommand{\aut}[1]{\operatorname{Aut}(#1)}
\newcommand{\inn}[1]{\operatorname{Inn}(#1)}
\newcommand{\out}[1]{\operatorname{Out}(#1)}

\newcommand{\Z}{\mathbb{Z}}

\newcommand{\BC}{\mathcal{BC}}

\newcommand{\centre}[2]{\operatorname{C}_{#1}(#2)}

\newcommand{\da}[1]{\operatorname{DA}_{#1}}

\newcommand{\edge}[1]{\operatorname{E}(#1)}
\renewcommand{\epsilon}{\varepsilon}

\newcommand{\Stab}[2]{\operatorname{Stab}_{#1}\left(#2\right)}
\newcommand{\dist}{\operatorname{d}}

\newcommand{\link}{\operatorname{Link}}

\renewcommand{\phi}{\varphi}

\newcommand{\ver}[1]{\operatorname{V}(#1)}

\title[JSJ splittings for all Artin groups]{JSJ splittings for all Artin groups}

\author[O. Jones]{Oli Jones}
    \address{(Oli Jones) School of Mathematics, Fry Building, Woodland Road\\ Bristol, UK}
    \email{oli.jones@bristol.ac.uk} 

\author[G. Mangioni]{Giorgio Mangioni}
    \address{(Giorgio Mangioni) Maxwell Institute and Department of Mathematics\\ Heriot-Watt University\\ Edinburgh, UK\\ Orcid: 0000-0003-2868-5032}
    \email{gm2070@hw.ac.uk}

\author[G. Sartori]{Giovanni Sartori}
    \address{(Giovanni Sartori) Maxwell Institute and Department of Mathematics\\ Heriot-Watt University\\ Edinburgh, UK}
    \email{gs2057@hw.ac.uk}

\begin{document}

\begin{abstract}
    We prove that an Artin group splits over infinite cyclic subgroups if and only if its defining graph has a separating vertex, and explicitly construct a JSJ decomposition over infinite cyclic subgroups for all Artin groups. We then use these facts to show that, if two Artin groups are isomorphic, then they have the same set of parabolics supported on big chunks, that is, maximal subgraphs without separating vertices. We also deduce acylindrical hyperbolicity for the automorphism groups of many Artin groups, partially answering a question of Genevois in the case of Artin groups. As a consequence, we produce new families of Artin groups with the $R_\infty$ property.
\end{abstract}

\maketitle

\small



\normalsize


\section*{Introduction}
Artin groups are a rich class of groups generalising braid groups, with strong connections to Coxeter groups. They are defined via the following presentation: given a finite simplicial graph $\Gamma$ with vertices $\ver\Gamma$, edges $\edge\Gamma$, and for each edge $\{a,b\} \in \edge{\Gamma}$ a label $m_{ab} \in \mathbb{Z}_{\geq 2}$, the associated Artin group $A_\Gamma$ is presented by
$$\langle \ver{\Gamma} \ | \ \mathrm{prod}(a,b,m_{ab})=\mathrm{prod}(b,a,m_{ab}) \, \forall \{a,b\}\in \edge{\Gamma}\rangle,$$
where $\mathrm{prod}(u,v,n)$ denotes the prefix of length $n$ of the infinite alternating word $uvuvuv\dots$. Despite the explicit presentation, very few properties are known to hold for \emph{all} Artin groups: it is not even known if they are all torsion-free, or satisfy the famous $K(\pi,1)$-conjecture \cite{Paris_Kpi1}. This is in stark contrast with the situation for Coxeter groups, whose algebraic and geometric features are well understood: they are linear over $\mathbb{R}$, hence they are virtually torsion-free, have solvable word problem, and are residually finite \cite{tits_solvable}; they admit a CAT(0) metric \cite{moussong}; and their torsion subgroups are conjugated inside finite parabolic subgroups \cite{bruhat-tits}.

In this paper we investigate when an Artin group splits over $\Z$, meaning that it admits a graph of group decomposition with infinite cyclic edge groups. Such splittings assumed a central role in the solution of the isomorphism problem for hyperbolic groups (see \cite{Sela_isoprob,Bowditch_cutpoints,RipsSela_cyclicsplit, DahmaniGuirardel} among many others); in particular, ``maximal" splittings called \emph{JSJ decompositions} were the key tool for understanding automorphisms of hyperbolic groups. In a similar fashion, we also obtain results related to the isomorphism problem for Artin groups, and establish some geometric properties of automorphism groups of Artin groups (see Theorem~\ref{thmintro:isoinvariance} and Theorem~\ref{thmintro:acylindrical}, respectively).

\subsection*{Cyclic splittings of Artin groups.}
Clay \cite{Clay} proved that a right-angled Artin group (RAAG) on a connected graph $\Gamma$ splits over $\Z$ if and only if either $\Gamma$ contains a separating vertex, or $\Gamma$ has two vertices (since $\Z^2$ splits over $\mathbb{Z}$). We generalise this to all Artin groups:
\begin{thmintro}[see Theorem~\ref{thm:no_split}]\label{thmintro:splitting}
    Let $\Gamma$ be a finite, connected, labelled simplicial graph. The Artin group $A_\Gamma$ splits over $\Z$ if and only if either $\Gamma$ has a separating vertex, or $\Gamma$ has two vertices. 
\end{thmintro}
In other words, if an Artin group on at least three generators splits over $\Z$, then it must have a \emph{visual} splitting, that is, a splitting coming from a decomposition of the graph into two full proper subgraphs overlapping over a vertex. We also have to include the case where $\Gamma$ consists of a single edge; indeed, such an Artin group (called a dihedral Artin group) splits over $\Z$, as it is a central extension of an infinite cyclic group by a virtually free group (see Example~\ref{ex:jsjSplitDihedral}).

Our proof of Theorem~\ref{thmintro:splitting} is similar in spirit to Clay's. However, we had to implement some upgrades, to circumvent those steps which used properties of RAAGs that are not true for all Artin groups.

\subsection*{Isomorphism invariance of big chunk parabolics.}
One of the most widely open questions for Artin groups is the isomorphism problem, which asks for an algorithm which takes two labelled graphs $\Gamma$ and $\Gamma'$ as input, and determines if they yield isomorphic Artin groups. This question, which specialises Dehn's Isomorphism Problem to the family of Artin groups, is solved when one restricts to certain subclasses, including RAAGs \cite{droms1987isomorphisms,baudisch1981raagcharacterise}, spherical-type \cite{Paris_isoprob_spherical}, and large-type \cite{vaskou_isoproblem}, but there is no unified approach. However, isomorphic Artin groups must admit the same splittings over $\Z$, as this is an algebraic property. In the same spirit, if two Artin groups are isomorphic, one can hope to use Theorem~\ref{thmintro:splitting} to recover a correspondence between ``maximal'' parabolic subgroups which do not split over $\Z$. 

To make the above more precise, we first need a definition. Given a labelled graph $\Gamma$, a \emph{big chunk} is a connected induced subgraph without separating vertices, which is maximal with these properties.\footnote{For unlabelled simplicial graphs, a big chunk is also called a \emph{biconnected component} or a \emph{block}. Though the latter terminology is more widespread, especially in computer science, we preferred the former, as it was recently used in the literature in the context of Artin groups.} In Theorem~\ref{thm:iso_invariance}, of which we give a streamlined version here, we prove that isomorphic Artin groups have isomorphic \emph{big chunk parabolics}, that is, parabolic subgroups supported on big chunks:
\begin{thmintro}[see Theorem~\ref{thm:iso_invariance}]\label{thmintro:isoinvariance}
     Let $\Gamma$ and $\Gamma'$ be finite, connected, labelled simplicial graphs, and let $\BC(\Gamma)$ and $\BC(\Gamma')$ be the collection of big chunks of $\Gamma$ and $\Gamma'$, respectively. Furthermore, let $\phi\colon A_\Gamma\to A_{\Gamma'}$ be an isomorphism. Then there exists a bijection $\phi_\#\colon \BC(\Gamma)\to \BC(\Gamma')$, such that for every $\Lambda\in \BC(\Gamma)$  the following hold:
     \begin{enumerate}
        \item $A_{\Lambda}\cong A_{\phi_\#(\Lambda)}$.
         \item $A_{\phi_\#(\Lambda)}$ is a conjugate of $\phi(A_{\Lambda})$, possibly unless $\Lambda$ is a leaf with label $2$.
          \item  If $\Lambda$ is an even leaf, then so is $\phi_\#(\Lambda)$, with the same label.
     \end{enumerate}
\end{thmintro}
\noindent Theorem~\ref{thmintro:isoinvariance} is already new for two-dimensional Artin groups, and we believe that our findings, together with \cite[Theorem E]{vaskou_isoproblem}, could be a key building block for a future solution of the isomorphism problem in this class. 

It should be noted that Theorem~\ref{thmintro:isoinvariance} alone is \emph{not} enough to reduce the isomorphism problem to the class of Artin groups whose defining graph has no separating vertices, as having isomorphic big chunks parabolic subgroups does not immediately imply that the Artin groups are isomorphic. One can still reduce the \emph{Strong Twist Conjecture}, a variation of the isomorphism problem formulated in terms of Artin generating sets, to Artin groups whose defining graphs have no separating vertices, which is precisely what we do in our subsequent work \cite[Theorem D]{JonManSar_StrTwiCon}. On the other hand, Theorem~\ref{thmintro:isoinvariance} provides an isomorphism invariant, hence a criterion to distinguish non-isomorphic Artin groups, (see Example~\ref{ex:distinguish_artin_via_thm} below for a concrete pair of Artin groups we can distinguish with our tools). If one solved the isomorphism problem for Artin groups on graphs without separating vertices, the criterion could then be checked algorithmically, as there is a procedure that, given a finite labelled simplicial graph, enumerates all its big chunks.

\begin{figure}[htp]
    \centering
    \includegraphics[width=0.75\linewidth]{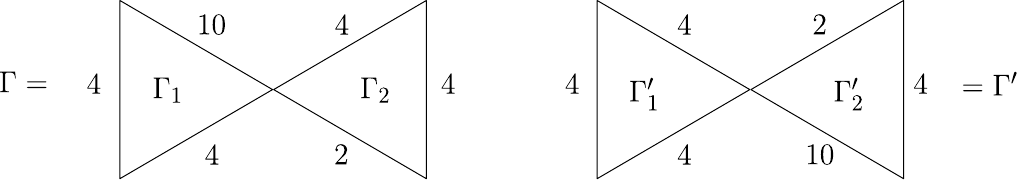}
    \caption{Theorem~\ref{thmintro:isoinvariance} implies that the groups $A_\Gamma$ and $A_{\Gamma'}$ are not isomorphic, as detailed in Example~\ref{ex:distinguish_artin_via_thm}.}
    \label{fig:bigChunk}
\end{figure}

\begin{exintro}\label{ex:distinguish_artin_via_thm}
    Let $\Gamma$ and $\Gamma'$ be the labelled graphs of Figure~\ref{fig:bigChunk}, let $\Gamma_1$ and $\Gamma_2$ denote the big chunks of $\Gamma$, and let $\Gamma_1'$ and $\Gamma_2'$ denote the big chunks of $\Gamma'$. Using Theorem~\ref{thmintro:isoinvariance}, if there were an isomorphism $\phi\colon A_\Gamma\to A_{\Gamma'}$, then this would induce a permutation $\sigma$ of $\{1,2\}$ and isomorphisms $\phi_i\colon A_{\Gamma_i}\to A_{\Gamma'_{\sigma(i)}}$, for $i\in\{1,2\}$. Necessarily the isomorphisms would be $A_{\Gamma_1}\cong A_{\Gamma_1'}$ and $A_{\Gamma_2}\cong A_{\Gamma_2'}$, as being of large type is an isomorphism invariant \cite{MV_characterising_LT}. However, isomorphic Artin groups of large type have twist equivalent defining graphs \cite[Theorem A]{vaskou_isoproblem}, while $\Gamma_1$ and $\Gamma_1'$ are not, a contradiction.

    Note that the two Artin groups could not be told apart via previously known methods. The graphs $\Gamma$ and $\Gamma'$ are both connected, have the same number of vertices and edges, and have the same set of labels, counted with multiplicity. They are even-labelled, so the group abelianisations coincide; however, the set of labels does not satisfy Blasco-García--Paris criterion for even Artin groups \cite[Theorem 5.1]{blasco2022isomorphism}. None of the graphs is of spherical type \cite[Theorem 1.1]{Paris_isoprob_spherical} nor of large type \cite[Theorem A]{vaskou_isoproblem}. Both graphs are two-dimensional, but they have the same set of pieces, and none contains Euclidean triangles \cite[Theorem E]{vaskou_isoproblem}. Finally, the girth and weighted girth of $\Gamma$ and $\Gamma'$ coincide \cite[Theorem C]{sartori}.
\end{exintro}

\subsection*{JSJ decompositions for Artin groups} 
The key tool in the proof of Theorem~\ref{thmintro:isoinvariance} is the existence of a \emph{JSJ decomposition} over cyclic subgroups, in the sense of \cite{GuirardelLevitt}. We refer to Definition \ref{def:jsj_tree} for the details; for our purposes, it is enough to recall that, given a graph of group decomposition $\mathcal G$ of a group $G$, with cyclic edge groups, if $\mathcal G$-vertex groups are elliptic in every $G$-action on a tree with cyclic edge stabilisers, then $\mathcal G$ is a JSJ decomposition. In other words, a JSJ decomposition should be thought of as a ``maximal'' splitting of $G$, which means that vertex groups cannot be further split ``in a canonical way''. Our methods allow us to build a JSJ decomposition over cyclic subgroups of an Artin group $A_\Gamma$, whose vertex groups roughly correspond to the big chunk parabolics of $\Gamma$:

\begin{thmintro}[see Theorem~\ref{thm:JSJ}]\label{thmintro:JSJ}
    Let $\Gamma$ be a finite, connected, labelled graph on at least three vertices. There is an $A_\Gamma$-tree $J(\Gamma)$ which is a JSJ-decomposition for $A_\Gamma$ over virtually cyclic subgroups.
\end{thmintro}

The construction of $J(\Gamma)$ is explicit, but we postpone it to Definition~\ref{defn:J_gamma} and only provide an example here (see Figure~\ref{fig:JSJ_from_intro}). Crucially, the proof of Theorem~\ref{thm:JSJ} uses that big chunk parabolics on at least three vertices do not split further over $\Z$, as a consequence of Theorem~\ref{thmintro:splitting}.

\begin{figure}[htp]
\centering
\includegraphics[width=\textwidth, alt={An example of the JSJ decomposition for an Artin group. There is one white vertex for every separating vertex, and one black vertex for every big chunk. Even leaves further split as HNN extensions (if the label is two), or amalgamated products otherwise.}]{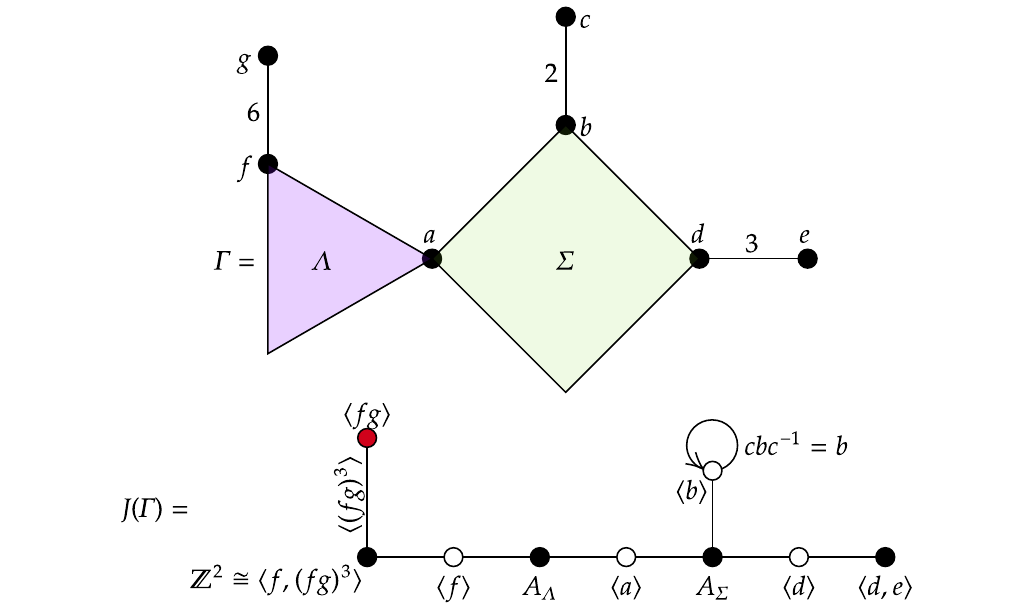}
\caption{The graph $\Gamma$ and JSJ decomposition $J(\Gamma)$ from Example \ref{exintro:jsj}.}
\label{fig:JSJ_from_intro}
\end{figure}

\begin{exintro}\label{exintro:jsj}
    Consider a connected graph $\Gamma$ as in Figure \ref{fig:JSJ_from_intro}, whose big chunks are $\Sigma$, $\Lambda$, and the three leaves, with labels $2$, $3$, and $6$, respectively. The associated JSJ decomposition $J(\Gamma)$ of $A_\Gamma$ is also depicted in the Figure. White vertices correspond to separating vertices of $\Gamma$, while black vertices roughly correspond to big chunks. The only exception are even leaves: a leaf with label $2$ (here, spanned by $b$ and $c$) further splits as a HNN extension, while an even leaf with higher label (here, generated by $f$ and $g$) splits as an amalgamation $\Z^2*_\Z \Z$ over its centre (here, generated by $(fg)^3$, as the label is $6$). Whenever an edge of $J(\Gamma)$ is not labelled in this picture, the edge group is the white vertex group, whose embedding in the black vertex group corresponds to the graph inclusion of a separating vertex inside a big chunk.
\end{exintro}

\begin{remintro}
    Our results above expand the (quite short) list of statements which hold for \emph{all} Artin groups, without extra requirements on the labels of the presentation graph. To the best of our knowledge, besides our work, the only such statements are that the intersection of standard parabolic subgroups is a standard parabolic subgroup \cite{VanDerLek}; that a parabolic subgroup $B$ contained in a parabolic subgroup $C$ is conjugated to a parabolic subgroup of $C$ \cite{blufstein2023parabolic}; the classification of which Artin groups are relatively hyperbolic~\cite{BDM}, coherent \cite{gordon2004artin, wise2013last}, or retract onto all of their parabolic subgroups \cite{retractable_artin}; and an algebraic characterisation of the girth of the presentation graph \cite{sartori}.
\end{remintro}

\subsection*{Acylindrical hyperbolicity of automorphism groups and consequences}
Our JSJ decomposition is not unique, as different Artin generating sets of the same Artin group may yield non-isomorphic JSJ decompositions. By work of Guirardel and Levitt \cite{guirardel2011cylinder}, there is a procedure that takes as input any JSJ tree and produces a canonical, $\aut{A_\Gamma}$-invariant tree, called the \emph{tree of cylinders}. Under suitable conditions on $\Gamma$, we can use the tree of cylinders to prove that $\aut{A_\Gamma}$ is acylindrically hyperbolic:

\begin{thmintro}[see Corollary \ref{cor:acylindrical}]\label{thmintro:acylindrical}
    Let $A_\Gamma$ be a torsion-free Artin group, such that $\Gamma$ is connected and has a separating vertex which is not a cone point of $\Gamma$ with each incident edge labelled by $2$. Then $\aut{A_\Gamma}$ is acylindrically hyperbolic.
\end{thmintro}

The assumption that $A_\Gamma$ is torsion-free is conjecturally trivial, since all Artin groups are expected to be torsion-free \cite[Conjecture 5.2]{McCammond:mysterious_geometry_artin_grp2019}. An Artin group as in Theorem~\ref{thmintro:acylindrical} is acylindrically hyperbolic, for instance by applying \cite[Corollary 2.2]{Minasyan_Osin} to the splitting over the separating vertex. In particular, our result provides further evidence for \cite[Question 1.1]{genevois2019negative}, which asks if every finitely generated acylindrically hyperbolic group has acylindrically hyperbolic automorphism group. In Remark~\ref{rem:artin_AH} we survey the progress on this question for Artin groups; as we explain there, to the best of our knowledge Theorem \ref{thmintro:acylindrical} provides the first non-right-angled examples where $\aut{A_\Gamma}$ is acylindrically hyperbolic and $\out{A_\Gamma}$ contains free subgroups.

An Artin group as in Theorem~\ref{thmintro:acylindrical} is also centerless, being torsion-free and acylindrically hyperbolic \cite{Osin_AH}. A centerless, infinite group with acylindrically hyperbolic automorphism group has several remarkable features, including the so-called \emph{$R_\infty$ property} (this follows by combining \cite[Corollary 8.1.4]{R_infty} and \cite[Theorem E]{FFF-Wade_Aut_invariant_quasimorph}, see also \cite[Corollary 3]{SorokoVaskou} for a recent independent proof):

\begin{corintro}\label{corintro:Rinfty}
    An Artin group $A_\Gamma$ as in Theorem~\ref{thmintro:acylindrical} has property $R_\infty$. 
\end{corintro}

Recall that, given a group $G$ and an automorphism $\phi\in \aut{G}$, two elements $g,h\in G$ are in the same \emph{$\phi$-twisted conjugacy class} if there exists $k\in G$ such that $\phi(k)gk^{-1}=h$. Furthermore, $G$ has property $R_\infty$ if every automorphism $\phi\in\aut{G}$ admits infinitely many twisted conjugacy classes. The $R_\infty$ property finds its origin in Nielsen fixed point theory, where under certain assumptions it relates to the number of fixed points of homeomorphisms of a space with fundamental group $G$. Among Artin groups, the $R_\infty$ property had been previously established for some spherical-type and affine-type Artin groups \cite{rinf_braid, rinf_purebraid,rinf_affine}, for several classes of large-type Artin groups \cite{Juahs_rinf,soroko2024propertyrinftynewclasses}, and for RAAGs \cite{rinf_raag1,rinf_raag2}.

We also list some other consequences of Theorem~\ref{thmintro:acylindrical}, which exhibit the existence of several non-trivial $\aut{A_\Gamma}$-invariant objects:

\begin{corintro}\label{corintro:cor_of_AH(Aut)}
    An Artin group $A_\Gamma$ as in Theorem~\ref{thmintro:acylindrical} admits:
    \begin{itemize}
        \item an infinite simple characteristic quotient \cite[Corollary 1.4]{charquot};
        \item an infinite-dimensional space of $\aut{A_\Gamma}$-invariant quasimorphism \cite[Theorem E]{FFF-Wade_Aut_invariant_quasimorph};
        \item unbounded $\aut{A_\Gamma}$-invariant word norms \cite[Corollary C]{FFF-Wade_Aut_invariant_quasimorph}.
    \end{itemize}
\end{corintro}

\subsection*{The JSJ deformation space}
The cost of the tree of cylinders is that the edge groups are no longer cyclic. If one wishes to consider cyclic splittings, the canonical object is the  \emph{JSJ deformation space}. Deformation spaces were first introduced by Forester~\cite{forester}, motivated by outer space for free groups~\cite{culler1986moduli}. Any JSJ decomposition of a group $G$ gives rise to the JSJ space~\cite{GuirardelLevitt}, a contractible space with an action of the (outer) automorphism group of $G$. As such, the JSJ deformation space turns out to be a powerful tool to study (outer) automorphism groups; indeed, in \cite{JonManSar_StrTwiCon}, we combine the results of this paper with the theory of deformation spaces to tackle a version of the isomorphism problem for new classes of Artin groups. More precisely, we are able to reduce the so-called Strong Twist Conjecture for an Artin group to its big chunk parabolics.

The first author recently defined another deformation space for certain large-type Artin groups~\cite{Jones_VF}. We note that, in general, their space does not coincide with the JSJ space of this article.

\subsection*{Future directions.}
In \cite{GrovesHull}, Groves and Hull generalised Clay's results and proved that a RAAG on a connected graph $\Gamma$ splits over an Abelian subgroup if and only if $\Gamma$ contains a separating simplex, or is itself complete, to account for $\Z^n$. It is possible that similar techniques may lead to an analogous characterisation of when an Artin group splits over a \emph{spherical} Artin subgroup. Notably, the analogue of Theorem~\ref{thmintro:isoinvariance} would yield the isomorphism invariance of \emph{chunk} parabolics, where a chunk is a maximal subgraph without separating cliques of spherical type. This would be relevant since the first author proved that, if the chunks of a large-type Artin group $A_\Gamma$ are isomorphism invariants and have finite outer automorphism group, then $\text{Out}(A_\Gamma)$ is of type VF, hence finitely presented \cite[Theorem 1.1]{Jones_VF}.

The first step in the above direction would be to understand the following:
\begin{question}\label{question:dihedral_split}
    Let $\Gamma$ be a finite, connected, labelled graph, and let $\da{n}$ be a dihedral Artin group with label $n$. What are necessary and sufficient conditions for $A_\Gamma$ to split over $\da{n}$?
\end{question}
The matter is complicated by the existence of ``exotic'' maximal dihedral subgroups which are not parabolic (see e.g. \cite[Theorem D]{vaskou_isoproblem}). Some of these questions will also be considered in independent work of Sliazkaite.

\subsection*{Organisation of the paper}
In Section~\ref{sec:background} we recall some properties of simplicial actions on trees. Section~\ref{sec:cyclicsplit} is devoted to the characterisation of when an Artin group splits over $\Z$, Theorem~\ref{thmintro:splitting}. In Section~\ref{sec:JSJ} we produce a JSJ decomposition, thus proving Theorem~\ref{thmintro:JSJ} (see Definition~\ref{defn:J_gamma} for the construction of the splitting, and Theorem~\ref{thm:JSJ} for the proof that it yields a JSJ decomposition). Then in Section~\ref{sec:AH} we combine the JSJ with the machinery of \cite{guirardel2011cylinder} to prove Theorem~\ref{thmintro:acylindrical} on acylindrical hyperbolicity of automorphism groups of certain Artin groups (see Corollary~\ref{cor:acylindrical}). Finally, the JSJ plays a central role in Section~\ref{sec:isoinvariance}, where we prove that big chunk parabolics are isomorphism invariant, Theorem~\ref{thmintro:isoinvariance} (see Theorem~\ref{thm:iso_invariance}).

\subsection*{Acknowledgements}
We thank Yassine Guerch and Gilbert Levitt for suggesting the application to acylindrical hyperbolicity of automorphism groups, and Francesco Fournier-Facio for pointing us towards its corollaries. We are also grateful to Kaitlin Ragosta for enlightenments on spherical Artin groups, to Bianca Marchionna for pointing us towards the connections to graph theory, and to Ray Sliazkaite for insightful discussions. A special thanks goes to our PhD supervisors Laura Ciobanu, Alessandro Sisto, and Alexandre Martin for comments on a first draft of this document. Finally, we are grateful to the anonymous referee for suggesting several ways to improve both the exposition and the arguments.

\section{Background}\label{sec:background}
\noindent We first recall some properties of simplicial actions on trees, referring to \cite{trees} for further generalities. We shall work in the following setting:
\begin{notation}\label{notation:action_on_trees}
    By \emph{tree} we mean a simply connected simplicial graph, equipped with the metric where each edge has length one. Given a group $G$, a \emph{$G$-tree} $(T,\Omega)$ is a tree $T$ endowed with an action $\Omega\colon G\to \aut{T}$ by simplicial isometries, without edge inversions, and minimal (i.e. no proper sub-tree is invariant under the action). We often suppress the reference to the action $\Omega$ when it is not relevant or it is clear from the context.
\end{notation}

\noindent Throughout we will write $\Stab{\Omega}{S}$ for the pointwise stabiliser of $S \subseteq T$, or $\Stab{G}{S}$ if there is no danger of confusing the action.

The \emph{translation length} of an element $g\in G$ is defined as $|g|\coloneq\inf_{x\in T}\dist_T(x,gx)$. The \emph{minset} of $g$ is the subtree spanned by all points $x$ which realise the translation length. If $|g|>0$ the element is \emph{loxodromic}, and its minset is a geodesic line on which $g$ acts by translations. If otherwise $|g|=0$ the element is called \emph{elliptic}, and its minset, also called the \emph{fixed subtree} of $h$ and denoted $\operatorname{Fix}_T(g)$, is the sub-tree of all fixed points of $g$. If $g$ is elliptic then so is every power of $g$, so one can define the \emph{stable fixed point set} of $g$ as the union $\operatorname{Fix}^{\infty}_T(g)=\bigcup_{n\in\mathbb N-\{0\}}\operatorname{Fix}_T(g^n)$.
 
The following lemmas are straightforward:

\begin{lemma}\label{lem:commuting_elm_and_minset}
Suppose a group $G$ acts on a tree $T$, and let $g$ and $h$ be commuting elements. Then the minset of $g$ is invariant under $h$.
\end{lemma}
\begin{proof}
    Write $T^s$ for the minset of $s \in G$. Then $T^g = T^{hgh^{-1}} = hT^g$, as required. 
\end{proof}

\begin{lemma}[{\cite[Corollary 2.5]{Clay}}]\label{lem:Z^2_no_fix}
    If $\Z^2$ acts on a tree without a global fixed point, then for any basis $\{g, h\}$, one of the elements must act loxodromically.
\end{lemma}

\noindent Recall also that a group $G$ \emph{splits} over a family $\mathcal{F}$ of subgroups if $G$ is isomorphic to the fundamental group of a \emph{graph of groups} whose edge groups are conjugates of subgroups in $\mathcal{F}$. If this is the case, then $G$ acts on the \emph{Bass-Serre tree} of the splitting, and edge stabilisers are conjugates of subgroups in $\mathcal{F}$. If moreover $G$ does not coincide with any vertex group (that is, if the splitting is non-trivial), then the action on the Bass-Serre tree has no global fixed point.

Conversely, if a group $G$ acts without global fixed points on a tree $T$, then $G$ is isomorphic to the fundamental group of a (non-trivial) graph of groups, whose edge groups are conjugates of edge stabilisers. With a little abuse of notation, we often conflate a $G$-tree with the corresponding graph of groups decomposition.

\section{Cyclic splittings of Artin groups}\label{sec:cyclicsplit}
\noindent Given a finite, labelled simplicial graph $\Gamma$, let $A_\Gamma$ be the associated Artin group. We fix once and for all an identification between $\Gamma$ and a generating set of $A_\Gamma$, and we will always conflate a vertex of $\Gamma$ with the corresponding group element. \par 
We recall here the notion of a parabolic subgroup. A subgraph $\Lambda\le\Gamma$ is said to be \emph{full}, or \emph{the full subgraph of~$\Gamma$ induced by $\ver\Lambda$}, if every edge of~$\Gamma$ connecting two vertices of~$\Lambda$ is also an edge of~$\Lambda$. For every full subgraph $\Lambda\le\Gamma$ we can consider the subgroup~$\langle{\ver\Lambda}\rangle_{A_\Gamma}$ of $A_\Gamma$ generated by $\ver\Lambda$. Such a subgroup is called the \emph{standard parabolic subgroup} of $A_\Gamma$ generated by $\ver\Lambda$ and is denoted $A_{\ver\Lambda}$. Then $A_{\ver\Lambda}$ is isomorphic to~$A_{\Lambda}$, by a result of Van der Lek~\cite[Theorem 4.13]{VanDerLek}, so henceforth we shall denote the standard parabolic subgroup on $\Lambda$ by $A_\Lambda$. A \emph{parabolic subgroup} of $A_\Gamma$ is a conjugate of a standard parabolic subgroup.

Recall that, if $\{a,b\}$ is an edge labelled with some integer $m\ge 3$, the \emph{dihedral Artin group} $\da{m}\coloneq\langle a,b\mid \mathrm{prod}(a,b,m)=\mathrm{prod}(b,a,m)\rangle$ has infinite cyclic centre, generated by 
\[z_{ab}=\begin{cases}
\Delta_{ab}     &\mbox{ if $m$ is even};\\
\Delta_{ab}^2   &\mbox{ if $m$ is odd},
\end{cases}\]
where $\Delta_{ab}=\mathrm{prod}(a,b,m)$ is the \emph{Garside element} of the dihedral Artin group (see e.g. \cite{BrieskornSaito}).

We now move to the characterisation of when an Artin group splits over $\Z$. Throughout the paper we actually work in the general setting of splittings over virtually cyclic subgroups. In fact every splitting we encounter is over $\Z$, which is unsurprising since Artin groups are conjecturally torsion-free while a virtually cyclic group which is not $\Z$ must have finite-order elements (see \cite[Lemma 3.2]{Macpherson}).

Recall that a vertex $v$ of a simplicial graph $\Gamma$ is \emph{separating} if the subgraph spanned by $\ver{\Gamma}-\{v\}$ is disconnected. Notice that, if $\Gamma$ is disconnected and has at least three vertices, then it has a separating vertex, because either it has at least three connected components and every vertex is separating, or it has at least one connected component which is not a point, all whose vertices are separating.

\begin{thm}\label{thm:no_split}
    Let $\Gamma$ be a finite, labelled simplicial graph. The Artin group $A_\Gamma$ splits over virtually cyclic groups if and only if
    \begin{itemize}
        \item $\Gamma$ has two vertices, or
        \item $\Gamma$ has a separating vertex.
    \end{itemize}
    In both cases, $A_\Gamma$ splits over $\Z$.
\end{thm}

\noindent The ``if'' part of Theorem~\ref{thm:no_split} is clear. Indeed, if $\Gamma$ has two vertices, the corresponding Artin group is isomorphic to either $\mathbb F_2$, $\Z^2$, or a dihedral Artin group, all of which split over $\Z$ (see  Example~\ref{ex:jsjSplitDihedral} for explicit splittings of dihedral Artin groups over $\Z$, and recall that the free group on two generators $\langle a,t\rangle$ splits as an HNN extension $\langle a,t,b\mid tat^{-1}=b\rangle$). Moreover, if $\Gamma$ has a separating vertex $v$, then there exist two proper induced subgraphs $\Gamma_1$ and $\Gamma_2$ of $\Gamma$ such that $\Gamma_1\cap\Gamma_2=\{v\}$, while $\Gamma_1\cup\Gamma_2=\Gamma$; therefore $A_\Gamma$ admits a \emph{visual splitting} $A_{\Gamma_1} *_{\langle v\rangle }A_{\Gamma_2}$. 

For the ``only if'' part, notice first that $\Z$ does not split non-trivially over any (virtually) cyclic subgroup: in any action on a tree, if any subgroup of $\Z$ is elliptic then the whole group fixes a point. Thus we shall devote the rest of the Section to proving that, if $\Gamma$ has at least three vertices and no separating vertices, then $A_\Gamma$ cannot split over virtually cyclic groups. It is useful to keep in mind that, by Bass-Serre theory, admitting no splitting over virtually cyclic groups is equivalent to the fact that, whenever the group acts on a tree without global fixed points, there exists an edge whose stabiliser is not virtually cyclic.

For the next lemma, recall that a \emph{Hamiltonian cycle} in a graph is an embedded cycle that visits each vertex exactly once.

\begin{lemma}\label{lem:hamcycle}
    Let $\Gamma$ be a finite simplicial graph admitting a Hamiltonian cycle. Then $A_\Gamma$ does not split over virtually cyclic groups.
\end{lemma}

\begin{proof}
    Let $T$ be a tree on which $A_\Gamma$ acts without global fixed points, and we want to find an edge whose stabiliser is not virtually cyclic. Let $C$ be the Hamiltonian cycle, and let $\{v_1,\ldots, v_n,v_{n+1}=v_1\}$ be a cyclic ordering of its vertices. Whenever $m_j\coloneq m_{v_jv_{j+1}}\ge 3$ let $z_j\coloneq z_{v_jv_{j+1}}$. Let $g_1=v_1$, and inductively define
    $$g_{i+1}=\begin{cases}
        v_{j+1}&\mbox{ if }g_i=v_j\mbox{ and }m_{j}=2;\\
        z_{j}&\mbox{ if }g_i=v_j\mbox{ and }m_{j}\ge 3;\\
        v_{j+1}&\mbox{ if }g_i=z_j.
    \end{cases}$$
    In other words, two consecutive $g_i$ and $g_{i+1}$ are either the vertices of an edge of $C$ with label $2$, or a vertex of an edge of $C$ with label at least $3$ and the generator of the corresponding centre. See Figure~\ref{fig:cycle_of_gens_and_centers} for an example.

\begin{figure}[htp]
        \centering
        \includegraphics[width=\textwidth, alt={Example of how to inductively construct $g_i$ so that every two consecutive $g$ generate a free Abelian group of rank $2$.}]{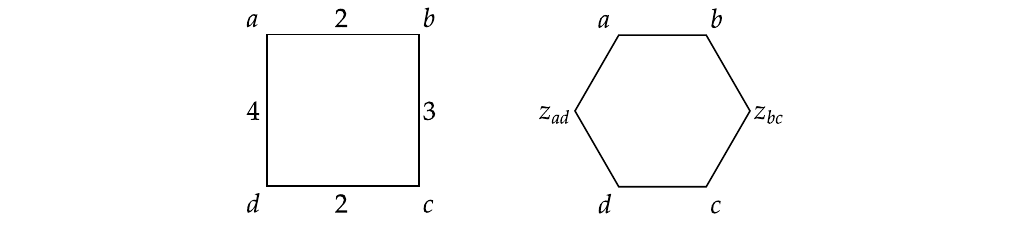}
        \caption{Here $\Gamma$ is the square on the left. Then we set $g_1=a$, $g_2=b$, $g_3=z_{bc}$, and so on. This way, every two consecutive $g_i$ generate a copy of $\Z^2$.}
        \label{fig:cycle_of_gens_and_centers}
    \end{figure}

    Now define $G_i\coloneq \langle g_i,g_{i+1}\rangle \cong \Z^2$. By construction $G_{i-1}\cap G_{i}=\langle g_i\rangle$ (this is a consequence of \cite{VanDerLek} if the subgroups are contained inside different dihedrals). There are two cases to analyse.
\par\medskip
\noindent \textbf{Case I}. Assume first that each $G_i$ fixes a point in $T$. If some $G_i$ fixes an edge of $T$ we are done; thus assume that each $G_i$ fixes a unique vertex $p_i$ of $T$. If all the $p_i$ were the same, then $A_\Gamma$, which is generated by the $G_i$s, would have a global fixed point. So let $p$ be a leaf of the subtree $S$ spanned by the $p_i$, and let $I\subseteq\{1,\ldots, r\}$ be such that $G_i$ fixes $p$ if and only if $i\in I$. Notice that $I$ is non-empty, and does not contain all indexes as there is no global fixed point; hence there exist $j,k\in I$ such that $j-1,k+1\not\in I$, where addition is modulo $r$. Let $p_{j-1}$ (resp. $p_{k+1}$) be the point fixed by $G_{j-1}$ (resp. $G_{k+1}$). This way, $g_j\in G_{j-1}\cap G_j$ fixes the non-trivial segment $[p_{j-1}, p]$, and similarly $g_{k+1}\in G_{k}\cap G_{k+1}$ fixes $[p_{k+1}, p]$. As $p$ was a leaf of $S$, the two segments share an edge, so $ \langle g_j, g_{k+1}\rangle$ fixes an edge. Now, $g_i$ and $g_{k+1}$ are infinite order elements without a non-trivial common power, as they are either two different vertices, or a vertex and the generator of a centre, or the generators of the centres of two different dihedrals (in all three cases, the absence of a common power follows immediately from \cite{VanDerLek}). In particular $\langle g_j, g_{k+1}\rangle$ is not virtually cyclic and fixes an edge, proving that $A_\Gamma$ does not split over virtually cyclic groups.
\par\medskip
\noindent \textbf{Case II}. Assume now that some $G_i$ fixes no point in $T$, say $G_1$ without loss of generality. By Lemma~\ref{lem:Z^2_no_fix}, we can assume without loss of generality that $g_2$ acts loxodromically on $T$, so it acts by translations on some geodesic line $A$. By Lemma~\ref{lem:commuting_elm_and_minset}, $g_1(A)=A$, so there exist integers $k_1,k_2$, with $k_1\neq 0$, such that $h=g_1^{k_1}g_2^{k_2}$ fixes $A$ pointwise. For the same reason, there exist integers $l_2,l_3$, with $l_3\neq 0$, such that $h'=g_2^{l_2}g_3^{l_3}$ fixes $A$ pointwise. Notice that, since $h\in G_1$, $h'\in G_2$, and $G_1\cap G_2=\langle g_2\rangle$, we get that $h$ and $h'$ are infinite order elements without a non-trivial common power. Hence, as in Case I, the fact that $\langle h, h'\rangle$ fixes an edge concludes the proof.
\end{proof}

\begin{lemma}\label{lem:NC_from_covering}
    Let $\Gamma$ be a finite, connected, labelled simplicial graph on at least three vertices, and let $\mathcal{L}$ be a collection of induced subgraphs of $\Gamma$ such that:
    \begin{enumerate}
        \item\label{item:nc_sub} For every $\Lambda\in \mathcal{L}$, $A_\Lambda$ does not split over virtually cyclic groups;
        \item\label{item:2edge_seg} Every two-edges segment of $\Gamma$ belongs to some $\Lambda\in \mathcal{L}$.
    \end{enumerate}
    Then $A_\Gamma$ does not split over virtually cyclic groups. 
\end{lemma}

\begin{proof}
    Suppose that $A_\Gamma$ acts on a tree $T$, without a global fixed point. If some $A_\Lambda$ has no fixed point in $T$, then by \eqref{item:nc_sub} a non-virtually-cyclic subgroup of $A_\Lambda$ (and therefore of $A_\Gamma$) fixes an edge, and we are done. Thus assume that each $A_\Lambda$ fixes a vertex of $T$. In particular, every vertex of $\Gamma$ fixes a point in $T$ (this is because $\Gamma$ is connected and has at least three vertices, so every vertex belongs to a two-edges segment). Furthermore, every two-edges segment with vertices $\{a,b,c\}$ fixes a vertex of $T$, which we can assume to be unique as otherwise the non-virtually-cyclic subgroup $\langle a,b\rangle$ would fix an edge. 

    Since there is no global fixed point, there exist vertices $v,v'$ of $\Gamma$ whose fixed point sets are disjoint. As argued above, $v$ and $v'$ cannot belong to the same two-edges segment, so they are at distance $n\ge 3$ in $\Gamma$. Furthermore, if we take a combinatorial path $v=v_0, \ldots, v_n=v'$ from $v$ to $v'$, there must be some index $0<i<n-1$ such that $\langle v_{i-1}, v_i, v_{i+1}\rangle$ and $\langle v_i, v_{i+1}, v_{i+2}\rangle$ fix different points in $T$. In particular, the dihedral $\langle  v_i, v_{i+1}\rangle$ fixes an edge, so $A_\Gamma$ does not split over virtually cyclic groups. 
\end{proof}

\noindent The following concludes the proof of Theorem~\ref{thm:no_split}:
\begin{lemma}\label{lem:NC_for_all}
    Let $\Gamma$ be a finite, connected, labelled simplicial graph on at least three vertices, without separating vertices. Then $A_\Gamma$ does not split over virtually cyclic groups.
\end{lemma}

\begin{proof}
    Let $\mathcal{L}$ be the collection of induced subgraphs of $\Gamma$ spanned by embedded cycles. For every $\Lambda\in \mathcal{L}$,  $A_\Lambda$ does not split over virtually cyclic groups by Lemma~\ref{lem:hamcycle}; moreover, since $\Gamma$ has at least three vertices and no separating vertex, every two-edges segment belongs to an embedded cycle. Hence $\mathcal{L}$ satisfies the requirements of Lemma~\ref{lem:NC_from_covering}, and the conclusion follows.
\end{proof}

\noindent We end the Section by reminding the reader of when an Artin group is one-ended, which we shall assume throughout the rest of the paper (see Notation~\ref{notation:3vert} below).

\begin{rem}\label{rem:one_end}
By Stallings' Theorem~\cite{stallings}, a group has more than one end if and only if it splits over finite subgroups. For an Artin group $A_\Gamma$, this happens if and only if $\Gamma$ is either a point or disconnected. Indeed, if $\Gamma$ is disconnected then $A_\Gamma$ is the free product of the parabolic subgroups on its connected components, while if $\Gamma$ is a point then $\Z$ is a HNN extension of the trivial group. Conversely, \cite[Proposition 1.3]{BDM} states that an Artin group on a connected graph on at least two vertices is not relatively hyperbolic, and in particular it cannot split over finite subgroups.
\end{rem}

\section{JSJ decomposition over virtually cyclic subgroups}\label{sec:JSJ}
\noindent This Section is devoted to the explicit construction of a JSJ decomposition over virtually cyclic subgroups for a one-ended Artin group. Recall that we work in the setting of Notation~\ref{notation:action_on_trees}, so all group actions on trees are assumed to be minimal and without inversions.
\begin{defn}[$VC$-splitting]
    Let $G$ be a group. A \emph{$VC$-splitting} of $G$ is a $G$-tree $(T,\Omega)$ such that, for every edge $e$ of $T$, $\Stab{\Omega}{e}$ is virtually cyclic.
\end{defn}

\begin{defn}[{JSJ tree, \cite{GuirardelLevitt}}]\label{def:jsj_tree}
     A $VC$-splitting $(T,\Omega)$ of a group $G$ is a \emph{JSJ tree over virtually cyclic subgroups} if:
     \begin{itemize}
         \item $T$ is \emph{universally elliptic}, meaning that, for every edge $e$ of $T$ and every $VC$-splitting $(T',\Omega')$, $\Stab{\Omega}{e}$ acts elliptically on $T'$;
         \item $T$ \emph{dominates} every universally elliptic tree, meaning that, if $T'$ is any universally elliptic $VC$-splitting, there is a $G$-equivariant map $T\to T'$; in other words, vertex stabilisers for $\Omega$ act elliptically on $T'$.
     \end{itemize}
     The associated graph of groups decomposition of $G$ is called a \emph{JSJ decomposition over virtually cyclic subgroups}.
\end{defn}

In fact, rather than checking the conditions in Definition~\ref{def:jsj_tree}, we will always use Lemma~\ref{lem:jsjSufficientCondition} to check that a given splitting is a JSJ decomposition. The hypothesis of Lemma~\ref{lem:jsjSufficientCondition} is strictly stronger than Definition~\ref{def:jsj_tree}; in the language of \cite{GuirardelLevitt}, it implies that our JSJ decompositions have no \emph{flexible} vertices.

\begin{lemma}\label{lem:jsjSufficientCondition}
    Suppose $(T,\Omega)$ is a VC-splitting of a group $G$ that dominates every other VC-splitting of $G$. Then $(T,\Omega)$ is a JSJ tree over virtually cyclic subgroups for $G$.
\end{lemma}
\begin{proof}
    The first point of Definition~\ref{def:jsj_tree} follows since each edge stabiliser of $(T,\Omega)$ is a subgroup of some vertex stabiliser, and the latter act elliptically in any other VC-splitting by hypothesis. The second point of Definition~\ref{def:jsj_tree} follows since $(T,\Omega)$ dominates every VC-splitting of $G$, so in particular the universally elliptic ones.
\end{proof}

\begin{ex}[JSJ splittings of dihedral Artin groups]
    \label{ex:jsjSplitDihedral}
    We first exhibit a JSJ decomposition $(T_n,\Omega_n)$ over virtually cyclic groups for a dihedral Artin group $\da{n}$. After constructing $(T_n,\Omega_n)$, distinguishing two cases according to the parity of~$n$, we shall prove that the latter is a JSJ decomposition using Lemma~\ref{lem:jsjSufficientCondition}. That is, we will show that, given any $VC$-splitting $(T,\Omega)$ for $\da{n}$, $\Omega_n$-vertex groups act elliptically on $T$. Notice first that edge stabilisers in $\Omega$ cannot be trivial, as dihedral Artin groups are one-ended by Remark~\ref{rem:one_end}; moreover, they must be infinite cyclic, since they are torsion-free (this follows from e.g. \cite[Theorem D]{CMV:parabolics}). 

    First, let $n$ be odd. Consider the splitting $(T_n,\Omega_n)$ as an amalgamated product induced by the isomorphism $\da{n}\cong\langle x,y\mid x^2=y^n\rangle$, where $y=ab$ and $x=ab\cdots a$, and let $z=x^2$. We claim that $z$ acts elliptically on $T$; this will imply that both $x$ and $y$, which are roots of $z$ and generate the $\Omega_n$-vertex groups, are elliptic on $T$, and therefore that $(T,\Omega)$ is a JSJ decomposition. Towards a contradiction, suppose instead that $z$ acts loxodromically on $T$, and let $A\subseteq T$ be its axis. Since $z$ generates the centre, every element in $\da{n}$ commutes with $z$, so it must fix $A$ by Lemma~\ref{lem:commuting_elm_and_minset}. In particular, by minimality of the action, the whole $T$ coincides with the axis of $z$. Let $g$ generate the stabiliser of an edge of $T$, and notice that, since the action is without inversions, $g$ fixes the whole $A$ pointwise. Then $\da{n}$ would be isomorphic to a semidirect product $\langle g\rangle\rtimes \langle z\rangle\cong \Z^2$, and this is not the case. 

    Now let $n=2m$ for $m\ge2$, and let $(T_n,\Omega_n)$ be the splitting as an HNN extension induced by the isomorphism  $\da{n}\cong\langle x,y\mid xy^mx^{-1}=y^m\rangle$, where $y=ab$ and $x=a$. Let $z=y^m$, which generates the centre. By the above argument $z$ must act elliptically on $T$, and in turn so does its root $y$, which generates the vertex group. Then again $(T_n,\Omega_n)$ is a JSJ tree over cyclic groups.

    Finally, there is no non-trivial JSJ decomposition over cyclic subgroups of $\Z^2$. Indeed, given any primitive element $a\in \Z^2$, complete it to a basis $\{a,b\}$ and consider the splitting as an HNN extension $\Z^2\cong \langle a,b\mid bab^{-1}=a\rangle$ where the vertex group is $\langle a\rangle$. This means that, if there were a JSJ decomposition over cyclic subgroups, then its edge stabilisers would be contained in every cyclic subgroup, so they would be trivial; this would contradict that $\Z^2$ is one-ended.
\end{ex}

\begin{notation}\label{notation:3vert}
    For the rest of this Section let $\Gamma$ be a finite, connected, labelled simplicial graph on at least three vertices, so that $A_\Gamma$ does not split over finite subgroups by Remark~\ref{rem:one_end}.
\end{notation}

\begin{defn}
    A \emph{big chunk} is a connected induced subgraph of $\Gamma$ without separating vertices, which is maximal (with respect to inclusion) with these properties. A \emph{big big chunk} is a big chunk on at least three vertices. A \emph{(big) big chunk parabolic} is a subgroup of $A_\Gamma$ conjugated to some $A_\Lambda$, where $\Lambda$ is a (big) big chunk.
\end{defn}
\noindent Notice that two big chunks can only overlap over a vertex, which must separate $\Gamma$; conversely, every separating vertex belongs to at least two big chunks. 

\begin{rem}[Big chunk parabolics are retracts]\label{rem:retraction}
Let $\Lambda$ be a big chunk of $\Gamma$. For every vertex $v\in \Gamma$ define $\rho(v)\in\Lambda$ to be the closest vertex in $\Gamma$ to $v$; such vertex is unique, because $\Lambda$ is a big chunk. Then mapping every generator $v$ to $\rho(v)$ gives a homomorphism $\rho\colon A_\Gamma\to A_\Lambda$ which is the identity on $A_\Lambda$. 
\end{rem}

\begin{defn}[$J(\Gamma)$]\label{defn:J_gamma}
Let $\Gamma$ be a finite, labelled, simplicial graph, and let $\mathcal{B}(\Gamma)$ be the bipartite graph defined as follows:\footnote{In graph theory, $\mathcal{B}(\Gamma)$ is known as the \emph{block-cut tree} of $\Gamma$.}
\begin{itemize}
    \item $\mathcal{B}(\Gamma)$ has one black vertex for every big chunk, and one white vertex for every separating vertex of $\Gamma$. We denote black vertices by the corresponding big chunks, and white vertices by the corresponding vertices of $\Gamma$.
    \item If $v$ is a white vertex and $\Lambda$ is a black vertex, there is an edge $\{v,\Lambda\}$ if and only if $v\in \Lambda$.
\end{itemize}

\noindent Next, recall that a \emph{leaf} of $\Gamma$ is an edge one of whose endpoints, called the \emph{tip} of the leaf, has valence one. A leaf is \emph{even} or \emph{odd} according to the parity of its label; an even leaf is \emph{toral} if the label is $2$, and is \emph{braided} otherwise. Now let $\mathcal{B}'(\Gamma)$ be the graph obtained from $\mathcal{B}(\Gamma)$ as follows:
\begin{itemize}
    \item glue a one-edge loop to each black vertex corresponding to a toral leaf;
    \item for every braided even leaf $\Lambda$, add a red vertex which is only adjacent to $\Lambda$.
\end{itemize}

\noindent We are finally ready to describe the candidate JSJ decomposition $J(\Gamma)$. The base graph of $J(\Gamma)$ is $\mathcal{B'}(\Gamma)$. Vertex and edge groups are as follows:
\begin{itemize}
    \item If $v$ is a white vertex, the associated vertex group is $\langle v\rangle$.
    \item If $\Lambda$ is a black vertex which is not an even leaf, the corresponding vertex group is $A_\Lambda$.
    \item Let $\langle a,b\rangle$ be a black vertex corresponding to a toral leaf with tip $b$. The associated vertex group is $\langle a\rangle$, and the stable letter of the loop is $b$, which commutes with $a$.
    \item Let $\langle a,b\rangle$ be a black vertex corresponding to an even leaf with label $2m\ge 4$ and tip $b$, and let $z=z_{ab}$. The black vertex group is $\langle a,z\rangle\cong \Z^2$; the red vertex group is generated by $r=ab$; and the edge between them identifies $r^m$ with $z$. To understand the construction, one should recall that
    $$\langle a,b\,|\,(ab)^m=(ba)^m\rangle\cong \langle a,z,r\,|\,[a,z]=1,\,r^m=z\rangle\cong \Z^2 *_{\Z} \Z.$$
    \item Whenever $\{v,\Lambda\}$ is an edge of $\mathcal{B}(\Gamma)$, the associated edge group is $\langle v\rangle$, and inclusions are given by subgraph containments.
\end{itemize}
\noindent The whole construction is clarified with an example in Figure~\ref{fig:expansion}.
\end{defn}

\begin{figure}[htp]
    \centering
    \includegraphics[width=\textwidth, alt={The various steps to construct the JSJ decomposition.}]{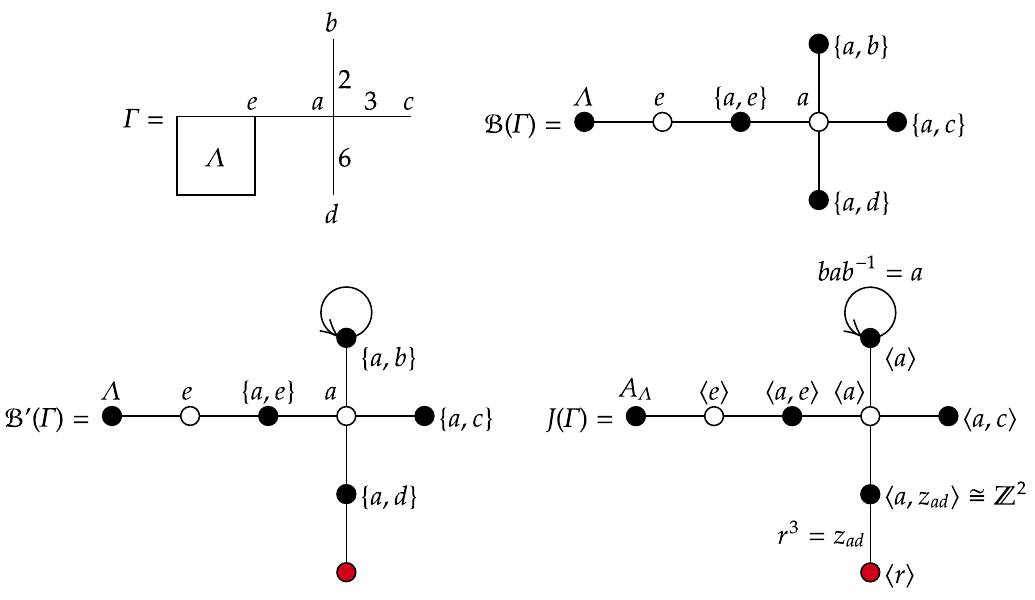}
    \caption{Example of the construction of the JSJ decomposition over virtually cyclic groups. First, one identifies the big chunks of $\Gamma$, which correspond to the black vertices of $\mathcal{B}(\Gamma)$. In this example, there are two even leaves, the toral leaf $\{a,b\}$ and the braided leaf $\{a,d\}$, with edge label $6$. Therefore, in $\mathcal{B}'(\Gamma)$ we glue a loop to $\{a,b\}$ and a free edge to $\{a,d\}$. Next, every white vertex is labelled with the subgroup generated by the corresponding vertex of~$\Gamma$. Every black vertex which is not an even leaf is labelled with the corresponding parabolic subgroup, while even leaves are further decomposed as HNN extensions (for toral leaves) or as amalgamated free products of $\Z^2$ and $\Z$ (for braided leaves). Every edge of $J(\Gamma)$ without a label represents the embedding of the white vertex group into the black vertex group, induced by inclusion of subgraphs of~$\Gamma$. }
    \label{fig:expansion}
\end{figure}

\begin{rem}[Minimality of $J(\Gamma)$]
    By construction, whenever $e$ is a leaf of $J(\Gamma)$ and $w$ is its tip, the edge group $H$ associated to $e$ properly embeds in the vertex group $V$ for $w$. This is because, if $w$ is black then $V$ is a big chunk parabolic, which is not cyclic; otherwise $w$ is red, $V=\langle r\rangle$ for some element $r$, and the edge group is generated by a proper power of $r$. By \cite[Lemma 7.12]{Bass}, this condition, together with the fact that the decomposition is finite, implies that $A_\Gamma$ acts minimally on the Bass-Serre tree of $J(\Gamma)$, so we are in the setting of Notation~\ref{notation:action_on_trees}.
\end{rem}

\begin{thm}\label{thm:JSJ}
    If $\Gamma$ is a connected graph on at least three vertices, and $(T, \Omega)$ is a $VC$-splitting, then every vertex group of $J(\Gamma)$ acts elliptically on $T$. As a consequence, $J(\Gamma)$ is a JSJ-decomposition for $A_\Gamma$ over virtually cyclic subgroups. 
\end{thm}

\begin{proof}
We first record three facts which we shall use multiple times throughout the proof of the theorem.

\begin{claim}\label{claim:vertex-fix}
    Let $\{a,b,c\}\subseteq\Gamma$ be a two-edges segment, with $b$ the middle vertex. Then $b$ acts elliptically on $T$. 
\end{claim}
\begin{claimproof}[Proof of Claim~\ref{claim:vertex-fix}]
    Let $g$ be either $a$ if $m_{ab}=2$, or $z_{ab}$ if $m_{ab}\ge 3$. Similarly, define $h$ as $c$ if $m_{bc}=2$, or $z_{bc}$ if $m_{bc}\ge 3$. If $b$ were loxodromic then one could find $k_1,k_2,l_1,l_2\in \mathbb{Z}$, with $k_1,l_1\neq 0$, such that the non-virtually-cyclic subgroup $\langle g^{k_1} b^{k_2}, h^{l_1} b^{l_2}\rangle$ would fix an edge, as in the proof of Lemma~\ref{lem:hamcycle}. 
\end{claimproof}

\noindent With the same techniques, one gets:
\begin{claim}\label{claim_dihedral_centre}
    Let $\{a,b\}$ be an edge of $\Gamma$, with $m_{ab}\ge 3$. Then $z_{ab}$ acts elliptically on $T$.
\end{claim}
\begin{claimproof}[Proof of Claim~\ref{claim_dihedral_centre}]
    Again, if $z_{ab}$ acted loxodromically, one could find $k_1,k_2,l_1,l_2\in \mathbb{Z}$, with $k_1,l_1\neq 0$, such that the non-virtually-cyclic subgroup $\langle a^{k_1} z_{ab}^{k_2}, b^{l_1} z_{ab}^{l_2}\rangle$ would fix an edge.
\end{claimproof}

\begin{claim}\label{claim:dihedral_nofix}
    Let $\{a,b\}$ be an edge of $\Gamma$. If $a$ and $b$ both act elliptically on $T$, then $\langle a,b\rangle$ fixes a point. 
\end{claim}

\begin{claimproof}[Proof of Claim~\ref{claim:dihedral_nofix}]
If $m_{ab}=2$ then $\langle a,b\rangle\cong \Z^2$ must act elliptically on $T$, as otherwise one of the generators should act loxodromically by Lemma~\ref{lem:Z^2_no_fix}. Then suppose that $m_{ab}\ge 3$, and let $z=z_{ab}$. By Claim~\ref{claim_dihedral_centre} $z$ acts elliptically on $T$, so let $T'$ be the minset of $z$. Notice that both $a$ and $b$ act on $T'$, by Lemma~\ref{lem:commuting_elm_and_minset}. Moreover, since both $a$ and $z$ act elliptically on $T$, the whole subgroup $\langle a,z\rangle$ fixes a point of $T$, by the contrapositive of Lemma~\ref{lem:Z^2_no_fix}; in other words, $a$ (and symmetrically $b$) must fix a point in $T'$.

Now, for every $n\in\Z-\{0\}$, the fixed subtree $\text{Fix}_{T'}(a^n)\subseteq T'$ is a single point $p_a$ not depending on $n$, because, if $\text{Fix}_{T'}(a^n)$ properly contained the non-trivial subtree $\text{Fix}_{T'}(a)$, then the non-virtually-cyclic subgroup $\langle a^n, z_{ab}\rangle$ would fix an edge (of $T'$, hence of $T$). As a consequence, $\text{Fix}_{T'}(a)$ coincides with the stable fixed point set $\{p_a\}=\text{Fix}_{T'}^\infty{(a)}$. A similar argument shows that $\text{Fix}_{T'}^\infty{(b)}=\text{Fix}_{T'}{(b)}$ is a single point $p_b$. Now suppose that $a$ and $b$ do not fix a common vertex of $T'$, and we claim that this leads to the contradiction that $a$ and $b$ generate a free group. Indeed, let $\gamma=[p_a,p_b]\subset T'$, and decompose $T'=X_a\cup \gamma\cup X_b$ where $X_a\cap \gamma=\{p_a\}$ and similarly $X_b\cap \gamma=\{p_b\}$. Since $\{p_a\}=\text{Fix}_{T'}^\infty{(a)}$, we see that $a^n(X_b)\subset X_a$ for every $n\in \Z-\{0\}$, and symmetrically with $a$ and $b$ swapped. Then the ping-pong lemma (see e.g. \cite{pingpong}) shows that $\langle a,b\rangle$ is non-Abelian free, a contradiction. 
\end{claimproof}

\noindent Going back to the proof, we want to show that every vertex group in $J(\Gamma)$ acts elliptically on $T$. By construction, every white vertex group of $J(\Gamma)$ is a subgroup of a black vertex group; moreover, if $\langle a,b\rangle$ is a braided even leaf, the associated red vertex group is generated by a root $r$ of $z_{ab}$, and in particular if $z_{ab}$ acts elliptically on $T$ then so does $r$. By this argument, we only need to consider black vertex groups. Furthermore, big big chunk parabolics must act elliptically on $T$, since by Lemma~\ref{lem:NC_for_all} they cannot split over virtually cyclic groups. So we are left to show that the vertex group associated to a big chunk on two vertices $\{a,b\}$ is elliptic. There are two cases to analyse.\medskip

\noindent \textbf{CASE I: $\{a,b\}$ is not a leaf.} In this setting both $a$ and $b$ are separating vertices, so there exist two edges $\{v,a\}$ and $\{b,w\}$ such that $a$ separates $v$ from $b$, and similarly $b$ separates $w$ from $a$. By Claim~\ref{claim:vertex-fix}, $a$ and $b$ both act elliptically on $T$. Then the black vertex group associated to $\{a,b\}$, which is the dihedral $\langle a,b\rangle$, acts elliptically on $T$ by Claim~\ref{claim:dihedral_nofix}.\medskip

\noindent \textbf{CASE II: $\{a,b\}$ is a leaf.} Let $b$ be the tip of the leaf. Since $\Gamma$ has at least three vertices and is connected, there exists $c\in \link_\Gamma(a)-\{b\}$. Applying Claim~\ref{claim:vertex-fix} to the two-edges segment $\{b,a,c\}$ shows that $a$ acts elliptically on $T$. There are now three cases to analyse, depending on $m_{ab}$.
\begin{itemize}
    \item If $m_{ab}=2$, the black vertex group associated to the leaf is $\langle a\rangle$, and there is nothing else to prove.
    \item If $m_{ab}>2$ is even then $z_{ab}$ acts elliptically on $T$ by Claim~\ref{claim_dihedral_centre}; thus the black vertex group associated to the leaf, which is $\langle a, z_{ab}\rangle\cong \Z^2$, acts elliptically on $T$ by Lemma~\ref{lem:Z^2_no_fix}.
    \item If $m_{ab}>2$ is odd then $b$ also acts elliptically on $T$, since it is conjugated to $a$; hence the black vertex group associated to the leaf, which is $\langle a,b\rangle$, fixes a point in $T$ by Claim~\ref{claim:dihedral_nofix}.
\end{itemize}
\noindent The proof of the first sentence of Theorem~\ref{thm:JSJ} is now complete. By Lemma~\ref{lem:jsjSufficientCondition}, this is sufficient to show that $J(\Gamma)$ is a JSJ decomposition over virtually cyclic subgroups.
\end{proof}

\section{Acylindrical hyperbolicity of automorphism groups}\label{sec:AH}
\noindent In this Section we prove that the automorphism group of an Artin group $A_\Gamma$ is acylindrically hyperbolic, provided the existence of a separating vertex of $\Gamma$ which is not central in $A_\Gamma$. To do so, from the JSJ decomposition from Section~\ref{sec:JSJ} we shall build an $\aut{A_\Gamma}$-invariant tree, which is an instance of the so-called tree of cylinders from \cite{guirardel2011cylinder}.

\subsection{Background on acylindrical hyperbolicity}
The notion of acylindricity is due to Sela, who originally formulated it for actions on trees, and was then extended by Bowditch to action on general metric spaces \cite{sela1997acylindrical,bowditch2008tight}. 

\begin{defn}[Acylindricity]\label{defn:acyl}
    Let $G$ be a group and let $(X,\dist)$ be a metric space. A $G$-action on $(X,\dist)$ by isometries is \emph{acylindrical} if, for every $\varepsilon\in\mathbb R_{\ge0}$ there exist $L\in\mathbb R_{\ge0}$ and $N\in\mathbb N$ such that, for every $x,y\in X$, if $\dist (x,y)\ge L$, then 
        \[
        \left|\{g\in G: \text{$\dist (x,gx)<\epsilon$ and $\dist (y,gy)<\epsilon$}\}\right|\le N.
        \]
In the case where $(X,\dist)$ is a simplicial tree endowed with the standard metric (i.e. edges are assigned unit length), the action is acylindrical if and only if there exist $L,N\in\mathbb N$ such that, for every $u,v\in\ver X$, if $\dist (u,v)\ge L$, then $\Stab{G}{u}\cap \Stab G v$ has order bounded by $N$ (this was Sela's original definition of acylindricity \cite[Section 2]{bowditch2008tight}).
\end{defn}

Recall that a geodesic metric space $X$ is \emph{hyperbolic} if there exists $\delta\ge0$ such that, for every geodesic triangle in $X$, each side is contained in the $\delta$-neighbourhood of the union of the other two. Notice that simplicial trees are hyperbolic with constant $\delta=0$.

\begin{defn}[\cite{Osin_AH}]
A group is \emph{acylindrically hyperbolic} if it is not virtually cyclic and it admits an acylindrical action on a hyperbolic metric space with unbounded orbits.
\end{defn}

\subsection{Trees of cylinders}
We now recall the construction of the tree of cylinders from \cite{guirardel2011cylinder}. For experts, we shall specialise the general definition to the case of virtually cyclic subgroups, with commensurability as the so-called admissible relation.
\begin{defn}[Cylinder]\label{defn:C_e}
        Let $(T,\Omega)$ be a $VC$-splitting for a group $G$. Given an edge $e$ of $T$, its \emph{cylinder} $C_e$ is the subforest of all edges $e'$ such that $\Stab{\Omega}{e}$ and $\Stab{\Omega}{e'}$ are commensurable (i.e. their intersection has finite-index in both). 
\end{defn}
By \cite[Lemma 4.2]{guirardel2011cylinder}, a cylinder is actually a subtree, so the intersection of two cylinders is connected. Furthermore, since cylinders are defined as equivalence classes of edges, they partition the edge set; combining these two facts, we see that two cylinders can overlap on at most a vertex. 

\begin{defn}[Tree of cylinders]\label{defn:T_c}
    The \emph{tree of cylinders} of $(T,\Omega)$ is the bipartite tree $T_c$ with vertex set $\ver{T_c}=V_0\sqcup V_1$ defined as follows:
    \begin{itemize}
        \item $V_0$ is the set of vertices $x$ of $T$ belonging to at least two distinct cylinders;
        \item $V_1$ is the set of cylinders of $T$;
        \item there is an edge $e=\{x,C\}$ between $x \in V_0$ and $C\in V_1$ if and only if $x$ (viewed as a vertex of $T$) belongs to $C$ (viewed as a subtree of $T$).
    \end{itemize}
    As $G$ acts on the set of cylinders, there is a natural isometric $G$-action on $T_c$ which we denote by $\Omega_c$. It is clear from the construction that $\Stab{\Omega_c}{x}=\Stab{\Omega}{x}$ for every $x\in V_0$; moreover, for every edge $e$ of $T$ we have that $\Stab{\Omega_c}{C_e}$ is the \emph{commensurator} of $\Stab{\Omega}{e}$ in $G$, i.e. the set of all $g\in G$ such that $\Stab{\Omega}{e}$ and $g\Stab{\Omega}{e}g^{-1}$ are commensurable.
\end{defn} 
It is easy to see that $T_c$ is indeed a tree \cite{gui04}; moreover, by \cite[Lemma 4.9]{gui04}, $\Omega_c$ is minimal if $\Omega$ is minimal. 

The following proposition is known to experts and included for completeness. The proof is largely based on personal communications with Yassine Guerch and Gilbert Levitt, both of whom we thank. 
\begin{prop}[see \cite{guirardel2011cylinder}]\label{prop:acyl_from_JSJ} 
Let $G$ be a group whose torsion subgroups have uniformly bounded cardinality. Let $(T,\Omega)$ be a JSJ-tree over virtually cyclic subgroups of $G$ which is not a point. Then:
\begin{itemize}
    \item $G$ acts minimally and acylindrically on $T_c$;
    \item The centre $Z(G)$ acts trivially;
    \item The $G/Z(G)$-action extends to $\aut{G}$.
\end{itemize}
\end{prop}

\begin{proof} We already noticed that the $G$-action on $T_c$ is minimal, so we now prove acylindricity by checking Sela's condition (see Definition~\ref{defn:acyl}). Given any $p,q\in \ver{T_c}$ at distance at least $6$, we shall prove that $\Stab{\Omega_c}{p}\cap \Stab{\Omega_c}{q}$ has uniformly bounded cardinality. Since the tree is bipartite, we can find $p',q'\in V_0$ which lie on a geodesic $[p,q]$ and are at distance at least $4$. Since any element fixing $p$ and $q$ must also fix $p$ and $q'$, it is enough to bound the cardinality of $\Stab{\Omega_c}{p'}\cap \Stab{\Omega_c}{q'}$. Since $p',q'\in V_0$, the latter equals $H\coloneq \Stab{\Omega}{p'}\cap \Stab{\Omega}{q'}$, which is therefore virtually cyclic as it is contained inside the stabiliser of some edge of $T$. If $H$ is finite then we are done, as torsion subgroups of $G$ have uniformly bounded cardinality. Otherwise $H$ is commensurable to all edge stabilisers on a $T$-geodesic between $p'$ and $q'$, which therefore belong to the same cylinder, violating the fact that $\dist_{T_c}(p',q')\ge 4$.

Moving to the second bullet, every element in the centre fixes every cylinder, as it maps edges of $T$ to edges with the same stabilisers. As every vertex in $V_0$ belongs to at least two cylinders, this means that the centre acts trivially on $T_c$. Then, if we identify $G/Z(G)$ with $\inn{G}\le \aut{G}$, the $G/Z(G)$-action on $T_c$ extends to $\aut{G}$ by \cite[Corollary 4.10]{guirardel2011cylinder}: for every $\phi\in\aut{G}$, the corresponding tree automorphism $\bar\phi\colon T_c\to T_c$ is the canonical $G$-equivariant isomorphism between $(\Omega_c,T_c)$ and the tree of cylinders of the twisted action $(\Omega\circ \phi, T)$.
\end{proof}

\subsection{Application to automorphism groups of Artin groups}

\begin{cor}\label{cor:acylindrical}
    Let $A_\Gamma$ be a torsion-free Artin group, such that $\Gamma$ is connected and has a separating vertex which is not a cone point of $\Gamma$ with each incident edge labelled by 2. Then $\aut{A_\Gamma}$ is acylindrically hyperbolic.
\end{cor}

\begin{proof}
Let $s\in\Gamma$ be a separating vertex that is not connected to every other vertex of $\Gamma$ via an edge labelled by~$2$.  In particular, there is a vertex $t\in \Gamma$ which generates either a non-Abelian dihedral Artin group or a free group with $s$. It is easy to see that $t$ does not commensurate $\langle s\rangle$, so $\langle s\rangle$ is \emph{weakly malnormal}, meaning that it has finite (in this case, trivial) intersection with one of its conjugates. Then, since $A_\Gamma$ splits as an amalgamated product over $\langle s\rangle$, it is acylindrically hyperbolic by \cite[Corollary 2.2]{Minasyan_Osin}. In turn, since it is also torsion-free, \cite[Corollary E]{bogopolski2022equations} provides an element $w$ such that, if an endomorphism $\phi\colon A_\Gamma\to A_\Gamma$ fixes $w$, then $\phi$ is the conjugation by a power of~$w$.

Now let $(T,\Omega)$ be the Bass-Serre tree of the JSJ splitting $J(\Gamma)$ from Definition~\ref{defn:J_gamma}, which is not a point since $\Gamma$ is not a single big chunk. By Proposition~\ref{prop:acyl_from_JSJ}, $A_\Gamma$ has a minimal, acylindrical action on the tree of cylinders $T_c$, which extends to $\aut{A_\Gamma}$. We take a small detour to prove the following:

\begin{claim}\label{claim:w_loxo}
    The element $w$ acts loxodromically on $T_c$, which is therefore unbounded.
\end{claim}

\begin{claimproof}[Proof of Claim~\ref{claim:w_loxo}] Towards a contradiction, suppose that $w$ fixes some vertex of $T_c$. If it fixes some $x\in V_0$ then $w\in \Stab{\Omega}{x}$, which is a vertex group of $T$ and is therefore contained in a big chunk parabolic $H$. However, by Remark~\ref{rem:retraction} there is a retraction $\rho\colon A_\Gamma\to H$, which is absurd as $\rho$ fixes $w$ and is not a conjugation.

Now assume that $w$ fixes the cylinder $C_e$, for some edge $e$ of $T$. We want to show that $w$ acts elliptically on $T$, from which we shall deduce a contradiction with the exact same argument as above.

Suppose first that $e$ connects a black and a red vertex, so that $\Stab{\Omega}{e}=g\langle z_{ab}\rangle g^{-1}$ for some braided even leaf $\{a,b\}$ and some $g\in A_\Gamma$. Up to replacing $w$ by $g^{-1}wg$, which has the same characterisation as $w$ in terms of endomorphisms, we can assume that $g=1$. Since $w$ commensurates $\Stab{\Omega}{e}$, there are non-trivial integers $m,n\in \Z-\{0\}$ such that $w z_{ab}^m w^{-1}=z_{ab}^n$. By looking at the retraction $\rho\colon A_\Gamma\to \langle a,b\rangle $, where $z_{ab}$ is central, we see that $m=n$, so $w$ and $z_{ab}^n$ commute. But then the conjugation $\psi$ by $z_{ab}^n$ fixes $w$, so $\psi$ must coincide with the conjugation by a power of $w$. Hence it must be that $w$ is a root of $z_{ab}^n$, since a torsion-free acylindrically hyperbolic group is centerless \cite[Corollary 7.2]{Osin_AH}. This means that $w$ acts elliptically on $T$ as one of its powers does, as required.

We are left to consider the case where $e$ is an edge whose stabiliser is conjugate to $\langle s\rangle$, for some separating vertex $s$. As above, up to conjugating $w$ we can assume that $w$ commensurates $\langle s\rangle$, so there are non-trivial integers $m,n\in \Z-\{0\}$ such that $w s^m w^{-1}=s^n$. By looking at the map $\rho\colon A_\Gamma\to \Z$ sending every generator to $1$ we see that $m=n$, so exactly as above we get that $w$ is a root of $s^n$, and in particular acts elliptically on $T$, as required. 
\end{claimproof}

Finally, to prove that $\aut{A_\Gamma}$ is acylindrically hyperbolic it is enough to verify the requirements of \cite[Proposition 3.8]{genevois2021acylindrical}:
\begin{itemize}
    \item $\aut{A_\Gamma}$ is not virtually cyclic as it contains $A_\Gamma$ (here we are again using that $A_\Gamma$ has no centre);
    \item The $A_\Gamma$-action on $T_c$ is acylindrical and minimal, hence orbits are unbounded as so is $T_c$. Thus the action is non-elementary by \cite[Theorem 1.1]{Osin_AH};
    \item $w$ is WPD, as it is a loxodromic element in an acylindrical action. 
    \item The subgroup $\{\phi \in \aut{G} \mid \phi(w) = w\}$ of $\aut{G}$ is precisely the cyclic subgroup generated by conjugation by $w$, as this is exactly the property $w$ was chosen for at the start of the proof.\qedhere
\end{itemize}
\end{proof}

\begin{rem}\label{rem:artin_AH} Corollary~\ref{cor:acylindrical} expands the known list of Artin groups with acylindrically hyperbolic automorphism groups, which we now survey.
\begin{itemize}
    \item Genevois and Horbez proved that infinitely-ended groups have acylindrically hyperbolic automorphism group \cite{genevois2021acylindrical}. This covers the automorphism groups of Artin groups on disconnected graphs, notably including $\aut{F_n}$.
    \item Building on the above work, Genevois characterised when a right-angled Artin group has acylindrically hyperbolic automorphism group, which is precisely when the RAAG itself is acylindrically hyperbolic \cite{genevois2024automorphisms}. 
    \item Irreducible spherical-type Artin groups have cyclic centre, so they are not acylindrically hyperbolic, but their central quotients are by \cite{calvez2017acylindrical}. The automorphism groups of Artin groups of type $A_n$ \cite{Dyer_Grossmann:Aut_braid1981} and $D_4$ \cite{soroko2021artin} are finite-index subgroups of mapping class groups of punctured hyperbolic surfaces, which are acylindrically hyperbolic by \cite{bowditch2008tight}, and the associated outer automorphism groups are finite. The Artin groups of type $D_n$ for $n\ge 5$ have finite outer automorphism groups \cite{castel2024endomorphismsartingroupstype, paris2026d5}, so their automorphism groups are virtually acylindrically hyperbolic.\footnote{We suspect they might be genuinely acylindrically hyperbolic, either by extending the acylindrical action from \cite{calvez2017acylindrical} to the full automorphism group, or, in the case $n\ge 6$, by extending the geometric representation of $A_{D_n}$ modulo its centre as a subgroup of a mapping class group to the full automorphism group.} The automorphism groups of Artin groups of type $B_n$, $I_{2n}$, and $F_4$ are virtually direct products of infinite groups \cite[Proposition 9]{charney2005automorphism}, so they are not acylindrically hyperbolic; on the other hand, the automorphism group of Artin groups of type $I_{2n+1}$ are  hyperbolic, since the outer automorphism groups are finite \cite{AutDihedral} and the central quotients are virtually free.  
    \item Euclidean-type Artin groups are acylindrically hyperbolic \cite{calvez2017acylindrical}. Automorphism groups of Artin groups of type ${\Tilde{A}}_n$ and ${\Tilde{C}}_n$ are also subgroups of finite index inside mapping class groups, and therefore acylindrically hyperbolic \cite{charney2005automorphism}; furthermore, their outer automorphism groups are finite.
    \item Large-type Artin groups are acylindrically hyperbolic \cite{vaskou2022acylindrical}, and their automorphism groups have been described for several sub-families \cite{crisp2005automorphisms, An_Hee:Automorphism_of_clttf, huang2024twistless, blufstein2024homomorphisms, Vaskou_LTFOI}. In all these classes, in the absence of separating vertices and edges, the outer automorphism group is finite, so the corresponding automorphism groups are virtually acylindrically hyperbolic.
\end{itemize}
In view of the above, Theorem \ref{thmintro:acylindrical} provides the first non-right-angled examples where $\aut{A_\Gamma}$ is acylindrically hyperbolic and $\out{A_\Gamma}$ contains free subgroups (for example, one can take $\Gamma$ to be two squares joined at a vertex, with all edges labelled by~3 \cite[Lemma 5.3.3]{jones2026thesis}).
\end{rem}

\section{Isomorphism invariance of big chunk parabolics}\label{sec:isoinvariance}
\noindent Using the JSJ decompositions we constructed, in this final Section we prove that, if two Artin groups are isomorphic, then any isomorphism must preserve the isomorphism type of big chunk parabolics. In fact, we show that every isomorphism preserves the conjugacy class of any big chunk parabolic that is not a toral leaf. \par 

We start with a lemma about dihedral Artin groups:
\begin{lemma}\label{claim:root_of_dihedral}
    Let $\langle a,b\mid (ab)^m=(ba)^m\rangle$ be a dihedral Artin group with label $2m$. If $x\in \langle a, z_{ab}\rangle$ is a primitive element admitting an $n$-th root $w\in \langle a,b\rangle$, then $n\le m$.
\end{lemma}

\begin{proof}
    Consider the presentation of the dihedral as the Baumslag-Solitar group $ \langle a,r\mid[a,r^m]=1\rangle\cong \Z*_{m\Z}$, where $r=ab$. In the Bass-Serre tree $T$ of this HNN extension, $\langle a,z_{ab}\rangle=\langle a, r^m\rangle$ is the setwise stabiliser of the axis $A$ of $a$. If $w$ acts loxodromically on $T$ then its axis, which is also the axis of $x=w^n\in\langle a,r^m\rangle$, must be fixed setwise by $a$, as a consequence of Lemma~\ref{lem:commuting_elm_and_minset}, so it must coincide with $A$. In this case $w$ fixes $A$ setwise, so $w\in \langle a,r^m\rangle$ against the assumption that $w$ was a non-trivial root. If instead $w$ is elliptic then so is $x$, and the only elliptic elements in $\langle a, r^m\rangle$ belong to $\langle r^m\rangle$ (this follows from Lemma~\ref{lem:Z^2_no_fix}). Since $x$ is primitive, we must have that $x=r^{\pm m}$, so $w$ has order $n$ in the central quotient $\langle a,b\rangle/\langle r^m\rangle\cong \Z*\Z/m\Z$. It now suffices to notice that torsion elements in the latter have order at most $m$.
\end{proof}

From now on, all defining graphs are finite, connected and on at least three vertices, as in Notation~\ref{notation:3vert}. We now introduce another cyclic splitting of an Artin group, which corresponds to the maximal visual splitting over separating vertices:
\begin{defn}\label{def:crushed_tree}
    Let $\overline{J(\Gamma)}$ be the splitting obtained from $J(\Gamma)$ by collapsing to a point every edge that either has a red vertex as an endpoint, or is a loop. See Figure~\ref{fig:ovJ} for an example of how to obtain $\overline{J(\Gamma)}$ from $J(\Gamma)$.
\end{defn}

\begin{figure}[htp]
\includegraphics[width=\textwidth, alt={The new splitting has one white vertex for every separating vertex and one black vertex for every big chunk. The vertex group of a black vertex is now the full big chunk parabolic, also for even leaves.}]{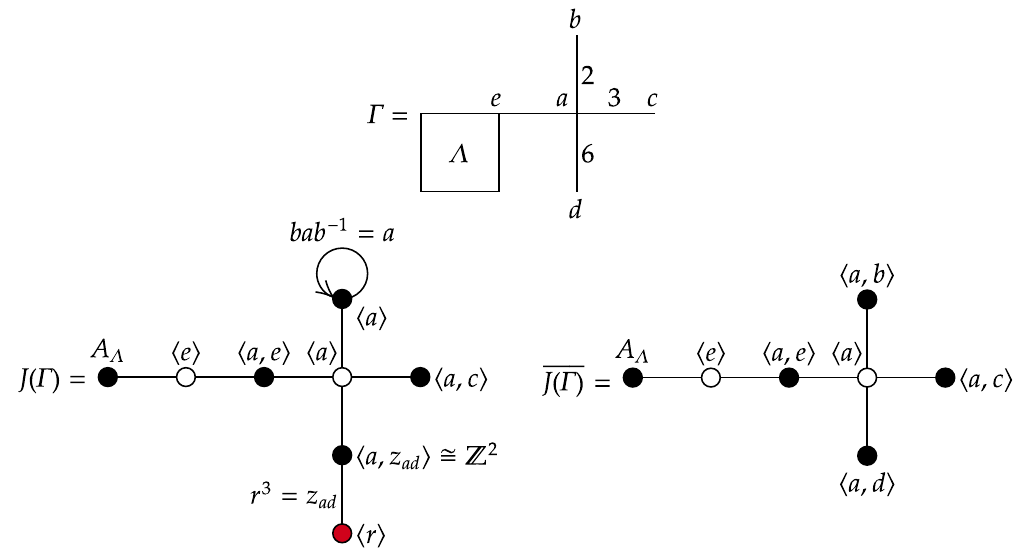}
\caption{$\overline{J(\Gamma)}$ has one white vertex for every separating vertex of $\Gamma$, and one black vertex for every big chunk. Now all vertices of $\Gamma$ act elliptically on the Bass-Serre tree of $\overline{J(\Gamma)}$, including the tips of even leaves.}
\label{fig:ovJ}
\end{figure}

\begin{lemma}\label{lem:nonconj_parabolics}
    Let $H,V$ be two standard big chunk parabolics of $A_\Gamma$. If $H\le kVk^{-1}$ for some $k\in A_\Gamma$, then $H=V$ and $k\in H$.
\end{lemma}

\begin{proof} In the Bass-Serre tree of $\overline{J(\Gamma)}$, the non-cyclic subgroup $H$ has a unique fixed point, which is the coset $H$; similarly, $kVk^{-1}$ only fixes $kV$. However $k V k^{-1}$ contains $H$, so $kVk^{-1}\le \Stab{A_{\Gamma}}{H}=H$. Hence $H=k V k^{-1}$, and therefore their fixed cosets $H$ and $kV$ coincide. In turn this means that $k\in V$, so $H=V$.
\end{proof}
\noindent
The first conclusion of Lemma \ref{lem:nonconj_parabolics}, namely that $H = V$, can also be deduced by combining \cite[Theorem 4.1]{Pparabolics} with \cite{blufstein2023parabolic}. The second, that $k \in H$, follows from the ribbon conjecture for Artin groups, which has been shown for some subclasses by Godelle but remains open in general \cite{godelle2003parabolic, godelle2007Artin}.

The exact same proof, with $J(\Gamma)$ replacing $\overline{J(\Gamma)}$, yields the following:

\begin{lemma}\label{lem:nonconj_vertex_groups}
    Let $H,V$ be two non-cyclic vertex groups of $J(\Gamma)$. If $H\le kVk^{-1}$ for some $k\in A_\Gamma$, then $H=V$ and $k\in H$.
\end{lemma}

\noindent We next notice that, since two JSJ decomposition have the same elliptic subgroups, then any isomorphism must preserve non-cyclic vertex groups:

\begin{prop}
\label{prop:phi(chunk)=chunk}
    Let $\phi\colon A_\Gamma\to A_{\Gamma'}$ be an isomorphism. Suppose that $H\le A_\Gamma$ is a non-cyclic black vertex group in $J(\Gamma)$. Then $\phi(H)$ is conjugate to a black vertex group of $J(\Gamma')$.
\end{prop}

\begin{proof}
   By Theorem~\ref{thm:JSJ}, $\phi(H)$ must act elliptically on the Bass-Serre tree for $J(\Gamma')$; thus there exists a vertex group $V$ of $J(\Gamma')$ and some $h\in A_{\Gamma'}$ such that $\phi(H)\le h Vh^{-1}$. Since $H$ is non-cyclic by assumption, $V$ itself must be a non-cyclic black vertex group of $J(\Gamma')$. Repeating this argument with $V$ and the inverse isomorphism $\phi^{-1}$, we get
    \begin{equation}\label{eq:containments_of_chunks}
        H=\phi^{-1}(\phi(H))\le \phi^{-1}(h) \phi^{-1}(V)\phi^{-1}(h)^{-1}\le k \widetilde H k^{-1},
    \end{equation}
    for some $k\in A_\Gamma$ and some non-cyclic black vertex group $\widetilde H$ of $J(\Gamma)$. By Lemma~\ref{lem:nonconj_vertex_groups} all inequalities in Equation~\eqref{eq:containments_of_chunks} are indeed equalities, so we must have had that $\phi(H)=h Vh^{-1}$.
\end{proof}

\noindent For the next Theorem, given a finite labelled graph $\Gamma$, let $\BC(\Gamma)$ be the set of big chunks of $\Gamma$.
\begin{thm}[Big chunk parabolics are isomorphism invariants]\label{thm:iso_invariance}
     Let $\Gamma$ and $\Gamma'$ be finite, connected, labelled simplicial graphs, and let $\phi\colon A_\Gamma\to A_{\Gamma'}$ be an isomorphism. Then there exists a bijection $\phi_\#\colon \BC(\Gamma)\to \BC(\Gamma')$ such that, for every $\Lambda\in \BC(\Gamma)$:
     \begin{enumerate}
        \item $A_{\Lambda}\cong A_{\phi_\#(\Lambda)}$.
        \item \label{item:conjugate} If $\Lambda$ is not a toral leaf, then $A_{\phi_\#(\Lambda)}$ is a conjugate of $\phi(A_{\Lambda})$.
        \item \label{item:evenleaf} If $\Lambda$ is an even leaf, then so is $\phi_\#(\Lambda)$, with the same label.
     \end{enumerate}
     Moreover, if $\phi$ maps standard generators of $\Gamma$ to conjugates of standard generators of $\Gamma'$, then we can arrange that $A_{\phi_\#(\Lambda)}$ is a conjugate of $\phi(A_{\Lambda})$ for every $\Lambda\in \BC(\Gamma)$.
\end{thm}

\begin{proof} By e.g. \cite[Corollary B]{MV_characterising_LT}, if $\Gamma$ has at most two vertices then $\Gamma$ and $\Gamma'$ are isomorphic as labelled graphs, so the result is trivial. Then assume that both $\Gamma$ and $\Gamma'$ have at least three vertices, so that we are in the setting of Notation~\ref{notation:3vert} and we can use the JSJ decompositions from Definition~\ref{defn:J_gamma}.

We shall progressively construct $\phi_\#$ on increasing subsets of $\BC(\Gamma)$, while checking at every step that conditions  \eqref{item:conjugate} and \eqref{item:evenleaf} are satisfied, and that the map is injective. Then the final map $\phi_\#$ will be injective as well, and satisfy the conditions. As a consequence, $\phi_\#$ will also be bijective, as the same procedure will also produce an injection $(\phi^{-1})_\#$ in the opposite direction.
\par\medskip

\noindent Let $\Lambda\in \BC(\Gamma)$, and let $H$ be the corresponding black vertex group in $J(\Gamma)$. Suppose first that $\Lambda$ is neither an even leaf nor an edge with label $2$, and let $\BC^+(\Gamma)$ be the collection of such big chunks. By construction, $H= A_\Lambda$, which is not isomorphic to either $\Z$ or $\Z^2$. By Proposition~\ref{prop:phi(chunk)=chunk}, $\phi(H)$ is conjugate to a black vertex group of $J(\Gamma')$, which must be of the form $A_{\Lambda'}$ for some $\Lambda'\in \BC^+(\Gamma')$ as $\phi(H)$ is also not isomorphic to either $\Z$ or $\Z^2$. Thus we can define $\phi_\#$ on $\BC^+(\Gamma)$ by setting $\phi_\#(\Lambda)=\Lambda'$. This map is injective because if $\Lambda,\Delta\in \BC^+(\Gamma)$ then $A_\Lambda$ and $A_\Delta$ are not conjugate by Lemma~\ref{lem:nonconj_parabolics}, so they map to non-conjugate vertex groups.
\par\medskip
\noindent Assume next that $\Lambda=\{a,b\}$ is an edge of label $2$ which is not a leaf, so that $H=A_\Lambda\cong \Z^2$. Again, Proposition~\ref{prop:phi(chunk)=chunk} provides some black vertex group $V\cong \Z^2$ of $J(\Gamma')$ such that $\phi(H)$ is conjugate to $V$. Suppose by contradiction that $V=\langle c, z_{cd}\rangle$ for some braided leaf $\{c,d\}$ with label $2m$. Then $z_{cd}=(cd)^m$, and since the isomorphism $\phi$ maps $H$ to $V$ there would be a primitive element $x\in H$ and some element $r\in A_\Gamma$ such that $r^m=x$. Now let $\rho\colon A_\Gamma\to H$ be the retraction from Remark~\ref{rem:retraction}, and notice that the element $\rho(r)\in H$ satisfies $\rho(r)^m=\rho(x)=x$, against the assumption that $x$ was primitive in $H$. 

By exclusion, we must have that $V=A_{\Lambda'}$ for some edge of label $2$ which is not a leaf, so we can extend $\phi_\#$ by setting $\phi_\#(\Lambda)=\Lambda'$. Again by Lemma~\ref{lem:nonconj_parabolics}, two different edges with label $2$ which are not leaves must support non-conjugate parabolics, so the extension of $\phi_\#$ is again injective.
\par\medskip
\noindent Now let $\Lambda=\{a,b\}$ be a braided even leaf with label $2m$ and tip $b$, so that $H=\langle a, z_{ab}\rangle$. By the above argument applied to the inverse of $\phi$, we see that $\phi(H)=kVk^{-1}$, where $k\in A$ and $V=\langle c, z_{cd}\rangle$ for some braided even leaf $\Lambda'=\{c,d\}$ with label $2n$ and tip $d$.

We first show that $m=n$. Suppose by contradiction that this is not the case, and up to replacing $\phi$ by its inverse we can assume that $m<n$. Since $\phi$ conjugates $H$ to $V$, to get a contradiction it is enough to show that a primitive element $x\in \langle a, z_{ab}\rangle$ cannot admit an $n$-th root $y\in A_\Gamma$. Let $\pi\colon A_\Gamma\to \langle a,b\rangle$ be the retraction from Remark~\ref{rem:retraction}. Then $\pi(y)\in \langle a,b\rangle$ is an $n$-th root of $x$, which by Lemma~\ref{claim:root_of_dihedral} means that $n\le m$, a contradiction.

We now claim that $\phi(\langle a,b\rangle)\le k\langle c,d\rangle k^{-1}$. If this is true, then the same argument applied to $\phi^{-1}$ will give that $\phi^{-1}(\langle c,d\rangle)$ is conjugated inside $ \langle a,b\rangle$, so 
$$\langle a,b\rangle=\phi^{-1}(\phi(\langle a,b\rangle))\le \phi^{-1}(k\langle c,d \rangle k)\le h\langle a,b \rangle h^{-1},$$
for some $h\in A$. Lemma~\ref{lem:nonconj_parabolics} will then imply that the above containments are indeed equalities, so that $\phi(\langle a,b\rangle)=k\langle c,d\rangle k^{-1}$. We will therefore set $\phi_\#(\Lambda)=\Lambda'$, and again invoke Lemma~\ref{lem:nonconj_parabolics} to get that the extension of $\phi_\#$ is injective.

To see that $\phi(\langle a,b\rangle)\le k \langle c,d\rangle k^{-1}$, let $\psi$ be the composition of $\phi$ and the conjugation by $k^{-1}$, so that we have to show that $\psi(\langle a,b\rangle)\le  \langle c,d\rangle$. In turn, it is enough to prove that $\psi(r)\le \langle c,d\rangle$ where $r=ab$. Since $r^m=z_{ab}$, the element $x\coloneq \psi(r)^m$ is primitive inside $\langle c,z_{cd}\rangle$; in particular, $\psi(r)$ fixes some vertex of the Bass-Serre tree for $J(\Gamma)$, which corresponds to some coset $gW$ where $g\in A_{\Gamma'}$ and $W$ is a vertex group of $J(\Gamma')$. Let $\mathcal{T}$ be the sub-tree of $J(\Gamma')$ corresponding to the $\langle c,d\rangle$-orbits of the edge between the black vertex $\langle c, z_{cd}\rangle $ and the red vertex $\langle cd\rangle$. If $gW\in \mathcal{T}$ we are done, because then both $g$ and $W$ belong to $ \langle c,d\rangle$ and so $\psi(r)\in gWg^{-1}\le \langle c,d\rangle$. Otherwise, $x$ fixes the geodesic from $gW$ to $\langle c,z_{cd}\rangle$, and on this geodesic let $e$ be the last edge not contained in $\mathcal{T}$. By how $J(\Gamma')$ is constructed, we must have that $e=h\langle c\rangle$, where $h$ lies in some $\langle c,d\rangle$-translate of $\langle c,z_{cd}\rangle$, and in particular in $\langle c,d\rangle$.  Then $x\in h\langle c\rangle h^{-1}\cap \langle c, z_{cd}\rangle$. Since no power of $h c h^{-1}$ is central in $\langle c,d\rangle$, the image $\overline{x}$ of $x$ in the central quotient $\langle c,d\rangle/\langle z_{cd}\rangle\cong \Z*\Z/m\Z$ is non-trivial, and lies in the intersection between $\langle \overline{c}\rangle$ and $\langle \overline{h}\overline{c}\overline{h}^{-1}\rangle$. Since $\langle \overline{c}\rangle$ generates the $\Z$-factor in the free product, which is malnormal, we have that $\overline{h}\in\langle\overline{c}\rangle$, so $h\in \langle c,z_{cd}\rangle$ and therefore $hch^{-1}=c$. In turn, $x\in\langle c\rangle$ is primitive, so $x=c^{\pm 1}$; however, sending every standard generator to $1$ defines a map $A_{\Gamma'}\to \Z$, which maps $\psi(r)$ to an $m$-th root of $\pm 1$, a contradiction.

\par\medskip
\noindent We are left to define $\phi_\#$ on toral leaves of $\Gamma$, whose number corresponds to the rank of the fundamental group of the graph underlying $J(\Gamma)$ (this rank is known as the \emph{Betti number} of $J(\Gamma)$). By a combination of~\cite[Theorem 1.1]{forester} and \cite[Section 4]{GuirardelLevitt}, the two decompositions $J(\Gamma)$ and $J(\Gamma')$ have the same Betti number (in fact, they belong to the same \emph{deformation space}). Thus we can fix any bijection $\psi$ between the set of toral leaves of $\Gamma$ and $\Gamma'$, and for every such leaf $\Lambda$ we set $\phi_\#(\Lambda)=\psi(\Lambda)$. This completes the construction of $\phi_\#$, which is injective as it was injective at every step. 
\par\medskip
In the setting of the ``moreover'' part of the statement we can modify the definition of $\phi_\#$ on toral leaves. Let $\{a,b\}$ be a toral leaf with tip $b$. Since $J(\Gamma)$ and $J(\Gamma')$ have the same elliptic subgroups, $\phi(b)$ must act loxodromically on the Bass-Serre tree for $J(\Gamma')$, so it must be conjugate to the tip $d$ of some even leaf $\{c,d\}$ because all other vertices of $\Gamma'$ act elliptically. Furthermore, such leaf must be toral, because $\phi$ preserves even braided leaves as we argued above. Hence let $k\in A$ be such that $\phi(b)=kdk^{-1}$. By looking at the action of $A_\Gamma$ on the Bass-Serre tree of the visual splitting $A_\Gamma\cong A_{\Gamma-\{a\}}*_{A_{a}} A_{\{a,b\}}$ one can see that $\operatorname{C}_{A_\Gamma}(b)=\langle a,b\rangle$: if $g\in A_\Gamma$ centralises $b$, then it stabilises the minset of~$b$, which consists of the single vertex $\langle a,b\rangle$. Analogously, $\operatorname{C}_{A_{\Gamma'}}(d)=\langle c,d\rangle$. It then follows that
\[
    \phi(\langle a,b\rangle)=\phi(\centre{A_\Gamma}{b})=\centre{A_{\Gamma'}}{\phi(b)}=k\centre{A_{\Gamma'}}{d}k^{-1}=k\langle c,d\rangle k^{-1};
\]
thus we can define $\phi_\#(\{a,b\})=\{c,d\}$, and again injectivity follows from Lemma~\ref{lem:nonconj_parabolics}.
\end{proof}

\bibliography{biblio.bib}
\bibliographystyle{alpha}
\end{document}